\documentclass[final,1p,times,authoryear]{elsarticle}
\usepackage{amssymb}
\usepackage{amsthm}
\usepackage{graphicx}%
\usepackage{multirow}%
\usepackage{amsmath,amssymb,amsfonts}%
\usepackage{mathrsfs}%
\usepackage[title]{appendix}%
\usepackage{xcolor}%
\usepackage{textcomp}%
\usepackage{manyfoot}%
\usepackage{booktabs}%
\usepackage{algorithm}%
\usepackage{algorithmicx}%
\usepackage{algpseudocode}%
\usepackage{listings}%
\usepackage{url}
\usepackage{makecell}

\newcommand{\vd}{\mathrm{d}}

\theoremstyle{definition}
\newtheorem{teo}{Theorem}

\newtheorem{lem}[teo]{Lemma}

\newtheorem{cor}[teo]{Corollary}
\newtheorem{rem}[teo]{Remark}

\journal{Nuclear Physics B}

\begin{document}
\begin{frontmatter}
%\fntext[a]{}
 \author[first]{Ubong Sam IDIONG\corref{cor1}}
 \affiliation[first]{organization={Department of Mathematics, Adeyemi Federal University of Education},
            address line ={143, Ondo-Ore Road},
            city={Ondo City},
           postcode={351103},
            state ={Ondo},
            country ={Nigeria}}
            \ead{usidiong@gmail.com,idiongus@aceondo.edu.ng}

%% \ead[url]{home page}
%\fntext[b]{}
\author[second]{Unanaowo Nyong BASSEY}
\affiliation[second]{organization={Department of Mathematics,},
 address line={University of Ibadan, Ibadan,  Nigeria},
           city={Ibadan},
       postcode={200005},
          state={Oyo State},
          country={NIGERIA}}
\ead{unbassey@yahoo.com}

%% \fntext[label3]{}

\title{Spectral Shift Functions of Lam\'e Operators}

\begin{abstract}
The search for spectral shift functions of operators remains an open area of research. In this paper, the Kre\u{\i}n's spectral shift functions are computed for the Lam\'e operator in the Weierstrass form and the Brioschi-Halphen operator through Green functions obtained by applying the technique of Fourier transform of distributions.
\end{abstract}

\begin{keyword}
Distributions\sep Gel'fand triple action\sep Generalized Function\sep Green functions

%% PACS codes here, in the form: \PACS code \sep code

%% MSC codes here, in the form: \MSC code \sep code
\MSC[2020] 46C20\sep 34M45

\end{keyword}

\end{frontmatter}

%\tableofcontents

%% \linenumbers

%% main text

\section{Introduction}
\label{introduction}
The expression gives the Lam\'e operator in the Weierstrass form
\begin{equation}\label{Lameop}
L_{s}:=-\;\frac{\vd^{2}}{\vd u^{2}}+s(s+1)\wp(u|g_{2},g_{3}), \qquad s\in\mathbb{R},
\end{equation}
where $\wp$ is the Weierstrass elliptic $\wp$-function that satisfies the polynomial equation $(\wp'(u))^{2}=4\wp(u)^{3}-g_{2}\wp(u)-g_{3}$, with $g_{2},g_{3}$ being invariant constants.
It is known by the works of Ince~\cite{INC} that the Lam\'e operator (considered on the real line
shifted by the imaginary half-periods $\omega_{k}(k=1,2,3)$ with $\sum_{k=1}^{3}\omega_{j}=0$ for integer $s$ has the property that its spectrum has exactly $s\in\mathbb{N}$ gaps which are such that $L_{s}:\mathscr{H}\rightarrow\mathscr{H}$ such that $\mathscr{H}:=L^{2}((0,\dfrac{\pi}{r}),\vd u)\subset \mathscr{C}_{c}^{\infty}(\Omega),\Omega=\mathbb{CP}^{1}\setminus\{e_{i}=\wp(\omega_{i})|i=1,2,3\}$ (see \cite{SNMR}, Eq.(20), p.426). The ends of the spectrum $E_{j}$ are given by the zeros of certain polynomials
\begin{equation}\label{sppol}
R_{2s+1}(E) = \prod_{i=0}^{2s}(E - E_{i}(s))=E^{2s+1}+b_{1}E^{2s}+ b_{2}E^{2s-1}+\ldots +b_{2s+1},
\end{equation}
with $b_{s}$ being functions of $s$. Here and hereafter, $E$ denotes the energy variable corresponding to the operator $L_{s}$ in \eqref{Lameop}. The polynomials given in \eqref{sppol} are called \emph{Lam\'e spectral polynomials} and were first studied by Hermite and Halphen (see~\cite{GV}, pp. 635-636 ).
The Brioschi-Halphen operator obtained by a two-step transformation of the Lam\'e operator \eqref{Lameop}  by setting $\psi=[\wp'(\frac{1}{2}u)]^{-s}\varphi$ and $w=\wp(u)$ is given as
\begin{equation}\label{bh}
H_{s}=4\prod_{i=1}^{3}(w-e_i)D^2-(2s-1)(6w^2-\frac{1}{2}g_2)D+4s(2s-1)w-4B,\;\;(D=\frac{\vd}{\vd w}),
\end{equation}
where $B$ is a constant, $s$ is restricted to be positive integer and the invariant constants $g_2$ and $g_3$ are given  by
\begin{equation}\label{invc}
g_2=4(e_1e_2+e_2e_3+e_1e_3)\hspace{.5cm} \textrm{and}\hspace{.5cm} g_3=4e_1e_2e_3
\end{equation}
where $e_j=\wp(\omega_j), j=1,2,3$ and each $\omega_j$ is the half-period of the Weierstrass elliptic
$\wp$-function.

In what follows, $\mathfrak{s}\ell(2,\mathbb{C})$ denotes the Lie algebra of complex two-by-two traceless matrices. It is also realised as a Lie algebra of differential operators generated by
\begin{equation}\label{gen}
\mathcal{J}_{+}:=w^{2}\frac{\vd}{\vd w}-2jw,\;\;\;\;\mathcal{J}_{0}:=w\frac{\vd}{\vd w}-j,\;\;\;\;
\mathcal{J}_{-}:=\frac{\vd}{\vd w}
\end{equation}
which obey the commutation relations
\begin{equation*}
[\mathcal{J}_{0},\mathcal{J}_{+}]=\mathcal{J}_{+},\;\;\;[\mathcal{J}_{-},\mathcal{J}_{+}]=2\mathcal{J}_{0},\;\;\;[\mathcal{J}_{-},\mathcal{J}_{0}]=\mathcal{J}_{-}\;\;.
\end{equation*}
(cf: \cite{TAV}, p.469).
The quantum Euler top operator obtained by $\mathfrak{s}\ell (2, \mathbb{C})$-algebraisation of Brioschi-Halphen operator is
\begin{equation}\label{rep}
\mathbf{H}_{s}:=4\mathcal{J}_{+}^2-g_2\mathcal{J}_{0}^2-\frac{g_3}{2}\mathcal{J}_{-}\mathcal{J}_{0}-\frac{3j-1}{4}g_2\mathcal{J}_{0}-\frac{3}{16}(3j-1)^{2}g_2\mathcal{J}_{-}
\end{equation}
with $j=\frac{3s}{2}$ being the spin. The operator $\mathbf{H}_{s}$ is found to be a Casimir operator with eigenvalues $E_{j}$ given as
$$E_{j}=\frac{j(j+1)g_{2}}{4}+\frac{3j(3j-1)}{4}.$$
Now, the characteristic polynomial of $\mathbf{H}_{s}$ coincides here with the Lam\'e spectral polynomial
\begin{equation}\label{det}
\det(E-\mathbf{H}_{s})=R_{2s+1}(E)=\prod_{j=0}^{2s}\left(E-\frac{j(j+1)g_{2}}{4}-\frac{3j(3j-1)}{4}\right).
\end{equation}
It is well known that the spectral shift function (SSF) is a function that calculates spectral flow and is associated with the spectra of operators. It is a key object in scattering theory and has applications in spectral theory, quantum mechanics, and statistical mechanics. The relationship between the determinant and the trace of the operator under consideration will be required in the development of the SSFs presented below. For an operator $X$ defined on a finite-dimensional space, the trace $Tr(X)$ and the determinant $\det(\exp(X)$ are related by the formula
\begin{equation}\label{trx}
Tr(X)=\ln(\det(\exp(X))).
\end{equation}
Similarly, the trace of the operator \eqref{det} can be written as
\begin{eqnarray*}
	Tr(\ln (E-\mathbf{H}_{s}))&=&\ln \det (\exp(\ln(E-\mathbf{H}_{s}))\\
	&=&\sum_{j=0}^{2s}\ln\left(E-\frac{j(j+1)g_{2}}{4}-\frac{3j(3j-1)}{4}\right).
\end{eqnarray*}
We note that when $s=1$ then $E=\wp(\varepsilon)$. In this case, we have
$$R_{3}(\wp(\varepsilon))=Tr(\ln (\wp(\varepsilon)-\mathbf{H}_{2}))=\sum_{j=0}^{2}\ln\left(\wp(\varepsilon)-\frac{j(j+1)g_{2}}{4}-\frac{3j(3j-1)}{4}\right).$$
Here, $\varepsilon$ is the pullback of the Weierstrass elliptic function given by the expression
$$\varepsilon=\wp^{-1}(E):=\int_{E}^{\infty}\frac{\vd w}{\sqrt{4w^{3}-g_{2}w-g_{3}}}.$$
The outline of the paper includes: Section~\ref{prel} contains essential preliminaries, Section\ref{sec5} contains the main results of this paper and Section~\ref{conc} contains the conclusion and summary of the paper.
\section{Preliminaries}\label{prel}
Let $\mathscr{B}(\mathscr{H})$ denote the space of all bounded linear operators defined on a Hilbert space $\mathscr{H}$, let $T\in\mathscr{B}(\mathscr{H})$ and let $R_{z}(T):=(z-T)^{-1}, z\in\mathbb{C}$ be the resolvent of the operator $T$. Following Kato (\cite{KT}, p. 44), assume $\varphi(\lambda)$ is an analytic ($=$holomorphic) function in a domain $\mathfrak{D}\subset\mathbb{C}$ containing all eigenvalues $\lambda_{n}$ of $T$ and let $\Gamma\subset \mathfrak{D}$ be a simple closed smooth curve with positive direction enclosing all the $\lambda_{n}$ in its interior. Then $\varphi(T)$ is defined by the Dunford-Taylor integral (Dunford~\cite{ND})
\begin{equation}\label{rcc1}
\varphi(T)=-\frac{1}{2\pi i}\int_{\Gamma}\varphi(\lambda)R_{\lambda}(T)\vd\lambda=\frac{1}{2\pi i}\int_{\Gamma}\varphi(\lambda)(\lambda-T)^{-1}\vd\lambda.
\end{equation}
This is an analogue of the Cauchy integral formula in function theory [see Knopp~\cite{KK}, p.61]. More generally, $\Gamma$ may consist of several simple closed curves $\Gamma_n$ with interiors $\mathfrak{D}_{n}'$ such that the union of the $\mathfrak{D}_n'$ contains all the eigenvalues of $T$.
%\begin{figure}
%\color{blue}  \centering
% \includegraphics[width=12cm, height=6cm]{rectcurve.jpg}
% \caption{The Domain of $\varphi(\lambda)$.}\label{rcc}
%\end{figure}

The resolvent set of $T$ is defined by $$\rho(T):=\{z|R_{z}(T)=(z-T)^{-1}\}.$$ Let $\sigma(T)$ denote the spectrum of  $T$ which is the defined as $\sigma(T):=\mathbb{C}\setminus \rho(T)$. An operator $T\in\mathscr{B}(\mathscr{H})$ is said to be compact if it takes a bounded sequence $(x_{n})$ to a sequence $(Tx_{n})$ with a convergent subsequence. The spaces $\mathscr{B}_{1}\mathscr{(H)}$ and $\mathscr{B}_{2}\mathscr{(H)}$ shall denote the set of trace-class and Hilbert-Schmidt operators respectively. We shall often have $H$ and $H_{0}$ as a pair of self-adjoint operators in with $\sigma(H), \sigma(H_{0})$ their spectra; $\rho(H), \rho(H_{0})$ their resolvent sets with $\mathcal{R}(z)$ and $\mathcal{R}_{0}$ their resolvents and $E_{\lambda}, E_{\lambda}^{0}$ the associated spectral families. The symbols $\|\cdot\|, \|\cdot\|_{1}, $ and $\|\cdot\|_{2}$ will denote operator norm, trace norm and Hilbert-Schmidt norm respectively, while $\mathrm{Tr} (B)$ will stand for the trace of a trace-class operator $B$. The functions of operators as defined by Kato shall be stated in what follows.

In what follows, the integration over the Riemann sphere which is the domain of the Lam\'e operator is discussed. This will assist us in the study of the compact nature of the Lam\'e integral operator. In this quest, we obtain a suitable measure for the integral relation. One of the important measures that exist for the Reimann sphere $\mathbb{CP}^{1}$ is the area measure of 1-form. In local coordinates,
\begin{eqnarray*}
% \nonumber % Remove numbering (before each equation)
  \vd z\wedge \vd \overline{z} &=& (\vd x+i\vd y)\wedge (\vd x-i\vd y) \\
  &=&\left|\frac{\partial(\vd z,\vd\overline{z})}{\partial(\vd x,\vd y)}\right|\vd x\vd y\\
  &=&\left|
       \begin{array}{cc}
         1 & i \\
         1 & -i \\
       \end{array}
     \right|\vd x\vd y
  \\
  &=&  -2i\vd x\vd y.
\end{eqnarray*}
Following Royden (\cite{HLR}, \S4, p.47), let $f\in\mathscr{C}^{1}(\mathbb{CP}^{1})$ be a piecewise smooth function over nets $\{\Omega_{n}\}_{n\in\mathbb{N}}$ in $\mathbb{CP}^{1}$ . Then if $w=\vd f$, one obtains the norm
$$0\leqslant \|w\|^2=\frac{i}{2}\int_{\Omega_{n}} w\wedge \overline{w}=\int_{\Omega_{n}}|w|^{2}\vd x\vd y<\infty. $$
It is also known (\cite{McM}, Theorem 7.1, p.72) that for any function $g\in\mathscr{C}_{c}^{\infty}(\mathbb{C}),$ a solution to the equation $\frac{\vd f}{\vd \overline{z}}=g$ is given by the convolution
$$f(z)=g\ast \frac{1}{\pi z}=\frac{i}{2\pi}\int_{\mathbb{C}}\frac{g(w)}{z-w}\vd w \wedge\vd\overline{w}.$$
\begin{teo}[\cite{KUT}, \S 9.10, p. 264] Let $f:X\times Y\longrightarrow [0,\infty]$ be measurable with respect to the $\sigma$-algebra $\sigma(\mathcal{K})$ and let $\mu\times \nu$ the product measure so that
\begin{equation}\label{d2}
(\mu \times \nu)(E)=\int_{Y} \int_{X} \chi_{E} \vd \mu\vd \nu=\int_{X} \int_{Y} \chi_{E} \vd \nu\vd \mu,
\end{equation}
where $\mu$ and $\nu$ are the finite measure on the the measurable spaces $(X,\mathcal{F})$ and $(Y,\mathcal{S})$ respectively. Then
\begin{equation}\label{d3}
\int_{X\times Y}f \vd(\mu \times \nu)=\int_{Y} \int_{X} f \vd \mu\vd \nu=\int_{X} \int_{Y} f\vd \nu\vd \mu,
\end{equation}
where $\mathcal{K}$ is the family of compact sets.
\end{teo}
\begin{proof} Since the measures are $\sigma$-finite there exists a pair of sequences $\{X_{n}\}$ and $\{Y_{n}\}$ such that $\nu(X_{n})<\infty$ and $\mu(Y_{n})<\infty$ then
\begin{equation}\label{d1}
\int_{Y_{n}} \int_{X_{n}} f \vd \mu\vd \nu=\int_{X_{n}} \int_{Y_{n}} f\vd \nu\vd \mu,
\end{equation}
so that as $n\rightarrow \infty$ equation \eqref{d1} gives
\begin{equation}
\int_{X\times Y}f \vd(\mu \times \nu)\equiv\int_{Y} \int_{X} f \vd \mu\vd \nu=\int_{X} \int_{Y} f\vd \nu\vd \mu,
\end{equation}
In particular if we take $f=\chi_{E}$ where $E\in\mathcal{F}\times\mathcal{S}$ we obtain \eqref{d2}.\qedhere
\end{proof}
\begin{lem}\label{Lem2} Let $m_k$ be the Lebesgue measure on $\mathbb{R}^{k}$, if we define a measure $\sigma_{k-1}$ on $\mathbb{S}^{k-1}$ (a unit sphere in $\mathbb{R}^{k}$). If further, $A\subset\mathbb{S}^{k-1}$ is a Borel set, let $\widetilde{A}:=\{r\mathbf{u}:0<r<1, \mathbf{u}\in A\}$ so that we define $\sigma_{k-1}(A)=k\cdot m_{k}(\widetilde{A}).$
Then
\begin{equation*}\displaystyle
\int_{\mathbb{R}^{k}} f\vd m_{k}=\int_{0}^{\infty}r^{k-1}\vd r \int_{\mathbb{S}^{k-1}}f(r\mathbf{u})\vd \sigma_{k-1}(\mathbf{u}).
\end{equation*}
is valid for every Borel function $f(\mathbf{x})\geq 0$ on $\mathbb{R}^{k}$ where, $\mathbf{x}=r\mathbf{u}.$
\end{lem}
\begin{proof} Without loss of generality, the Borel set $A$ is compact.We consider the sequence $\left(T_{n}\right)$ of measurable sets defined by
\begin{eqnarray*}\displaystyle
	T_{1} &=& \left\{\mathbf{x}\in A\bigg|\frac{1}{2}\leq \varphi(\mathbf{x}) <1\right\} \\
	T_{n} &=& \left\{\mathbf{x}\in A\bigg|\frac{1}{2^{n}}\leq \varphi(\mathbf{x})-\sum_{i=1}^{n-1}\frac{1}{2^{i}}\chi_{T_{i}}(\mathbf{x}) <\frac{1}{2^{n-1}}\right\},\qquad \forall n\geq 2 ,\end{eqnarray*}
where $\displaystyle \chi_{T_{i}}(\mathbf{x})$ is the characteristic function with respect to terms $T_{i}$, of the the sequence $\left(T_{n}\right)$ and $\varphi:\mathbb{R}^{k}\longrightarrow \mathbb{R}:\mathbf{x}\mapsto\varphi(\mathbf{x})$. . Also arguing by induction, it could be seen that for any $\mathbf{x}\in A, i\in \mathbb{N},\chi_{T_{i}}(\mathbf{x})=a_{i},$ where $a_{i}$ is the $i-$th digit in the binary expansion of $\varphi(\mathbf{x})$, i.e., $0.a_{1}a_{2}\ldots a_{i} \ldots$. Therefore,
\begin{equation*}
0\leq \varphi(\mathbf{x})-\sum_{i=1}^{n}\frac{1}{2^{i}}\chi_{T_{i}}(\mathbf{x}) <\frac{1}{2^{n}}, \forall \mathbf{x}\in\mathbb{R}^{N}, \forall n\in\mathbb{N}.
\end{equation*}
Now, as $n \rightarrow \infty$ we see that
\begin{equation}\label{d5}\displaystyle
\varphi(\mathbf{x})=\sum_{i=1}^{\infty}\frac{1}{2^{i}}\chi_{T_{i}}(\mathbf{x}), \forall \mathbf{x}\in\mathbb{R}^{N},
\end{equation}
where the series converges uniformly to 1 in $\mathbb{R}^{N}.$

Now to prove Lemma~\ref{Lem2} we see that geometrically, $\mathbb{R}^{k}-\{0\}= \left(0,\infty\right)\times \mathbb{S}^{k-1}$. We also need to define a function $f:\mathbb{R}^{k}\longrightarrow\mathbb{R}^{+}:r\mathbf{u}\mapsto f(r\mathbf{u}) $ .We shall construct the Borel measure or the volume differential element on $\mathbb{S}^{k-1}$ in what follows.

Firstly, we consider the $k-$dimensional hypersphere
\[\mathcal{H}^{k}:=\{\mathbf{x}\in\mathbb{R}^{k}:\|\mathbf{x}\|^{2}=\sum_{i=1}^{k}x_{i}^{2}\leq R^{2}\}.\]
The volume of the $k-$dimensional Hypersphere is
\[m_{k}(R)=\int\ldots\int_{\mathcal{H}^{k}}\vd x_{1}\ldots \vd x_{k}=C_{k}R^{k}.\]
It is also obvious that the boundary $\partial \mathcal{H}^{k}=\mathbb{S}^{k-1}$. Thus, we now see that if we add infinitely many thin spherical shells of radius $0\leq r \leq R$ we have
\[m_{k}(R)=\int_{0}^{R}\sigma_{k-1}(r)\vd r,\]
where $\sigma_{k-1}(R)$ is the surface area of $\mathcal{H}^{k}$ and by fundamental theorem of calculus
\[\sigma_{k-1}(R)= \frac{\vd m_{k}(R)}{\vd R}=kC_{k}R^{k-1},\]
which now gives
\[\vd m_{k}(R)=S_{k-1}(R)\vd R=R^{k-1}\vd R \quad kC_{k}=R^{k-1}\vd R \vd \sigma_{k-1}(R),\]
where $\vd \sigma_{k-1}(R)=kC_{k}$ is the differential element (or measure) on the sphere $\mathbb{S}^{k-1}$ with radius R,
so that by \eqref{d3} and \eqref{d5} we obtain
\begin{eqnarray}\label{d4}
\int_{\mathbb{R}^{k}}f(\mathbf{x})\vd m_{k}(\mathbf{x}) &=& \int_{[0,\infty)\times \mathbb{S}^{k-1}}f(\mathbf{x})\vd m_{k}(\mathbf{x})\nonumber \\
&=&\int_{[0,\infty)\times \mathbb{S}^{k-1}}f(r\mathbf{u})\vd m_{k}(r\mathbf{u}) \nonumber \\
&=& \int_{0}^{\infty}r^{k-1}\vd r\int_{\mathbb{S}^{k-1}}\sum_{i=1}^{\infty}\frac{1}{2^{i}}\chi_{T_{i}}(\mathbf{x})f(r\mathbf{u})\vd \sigma_{k-1}(\mathbf{u})\nonumber\\
&=& \int_{0}^{\infty}r^{k-1}\vd r\int_{\mathbb{S}^{k-1}}f(r\mathbf{u})\vd \sigma_{k-1}(\mathbf{u}).\hspace{2 in.}
\end{eqnarray}
However if $R=1$, $\mathcal{H}^{k} (\mathbb{S}^{k-1})$ are regarded as unit Hypersphere (unit Sphere) respectively. \qedhere
\end{proof}
\cor If the function $f(\mathbf{x})=e^{-\langle \mathbf{x},\mathbf{x}\rangle}$ in the equation \eqref{d4} we obtain $\displaystyle \sigma_{k-1}=\frac{k\pi^{\frac{k}{2}}}{\Gamma\left(1+\frac{k}{2}\right)}$.\ \\
\begin{proof} According to \eqref{d4}
\begin{eqnarray*}
	\int_{\mathbb{R}^{k}}f(\mathbf{x})\vd m_{k}(\mathbf{x})&=&\int_{\mathbb{R}^{k}}e^{-\langle \mathbf{x},\mathbf{x}\rangle}\vd m_{k}(\mathbf{x})\\
	&=&\int_{0}^{\infty}r^{k-1}\vd r\int_{\mathbb{S}^{k-1}}f(r\mathbf{u})\vd \sigma_{k-1}(\mathbf{u})  \\
	&=& \int_{0}^{\infty}r^{k-1}\vd r\int_{\mathbb{S}^{k-1}}e^{-\langle r\mathbf{u},r\mathbf{u}\rangle}\vd \sigma_{k-1}(\mathbf{u}) \\
	&=& \int_{0}^{\infty}r^{k-1}\vd r\int_{\mathbb{S}^{k-1}}e^{-r^{2}\langle \mathbf{u},\mathbf{u}\rangle}\vd \sigma_{k-1}(\mathbf{u}) \\
	&=&  \int_{0}^{\infty}r^{k-1}e^{-r^{2}}\vd r\int_{\mathbb{S}^{k-1}}\vd \sigma_{k-1}(\mathbf{u}) \\
	&=&\frac{1}{2}\Gamma\left(\frac{k}{2}\right)\sigma_{k-1}\\
	&=&\frac{k}{2}\Gamma\left(\frac{k}{2}\right)C_{k}\\
	&=&\Gamma\left(1+\frac{k}{2}\right)C_{k}.
\end{eqnarray*}
Also on the other hand we see that
\begin{eqnarray*}\displaystyle
	\int_{\mathbb{R}^{k}}f(\mathbf{x})\vd m_{k}(\mathbf{x}) &=& \int_{-\infty}^{\infty}\ldots\int_{-\infty}^{\infty}e^{-(x_{1}^{2}+\ldots +x_{k}^{2})}\vd x_{1}\ldots \vd x_{k} \\
	&=& \int_{-\infty}^{\infty}e^{-x_{1}^{2}}\vd x_{1}\ldots\int_{-\infty}^{\infty}e^{-x_{k}^{2}}\vd x_{k} \\
	&=& (\sqrt{\pi})^{k}\\
	&=& \pi^{\frac{k}{2}}.
\end{eqnarray*}
Combining the two results above we have that
\[\Gamma\left(1+\frac{k}{2}\right)C_{k}=\pi^{\frac{k}{2}}.\]
So that solving for $C_{k}$ we obtain
\[C_{k}=\frac{\pi^{\frac{k}{2}}}{\Gamma\left(1+\frac{k}{2}\right)}.\]
Hence, we obtain
\[\sigma_{k-1}=kC_{k}=\frac{k\pi^{\frac{k}{2}}}{\Gamma\left(1+\frac{k}{2}\right)}.\]
\end{proof}
\begin{rem} The value for the surface area of the sphere $\mathbb{S}^{k-1}\subset\mathbb{R}^{k}$ when $k=2,3$ are $2\pi$ and $4\pi$ respectively. We also notice the inclusion map $\mathbb{S}^{1}\hookrightarrow \mathbb{CP}^{1}$ so that the torus $\mathbb{T}=\mathbb{CP}^{1}/\mathbb{S}^{1}$ which is a quotient with $ker \mathbb{T}=\mathbb{S}^{1}$. So that there exist an equivalence relation between $\mathbb{CP}^{1}$ and $\mathbb{S}^{1}$. Hence, the integral over the Riemann sphere is equivalent to the integral over $\mathbb{S}^{1}$.
\end{rem}
It is known (\cite{AMP1}, \S 16.1, Examples 4 of Manifold p.187) that the invariant measure on the 2-dimensional sphere $\mathbb{S}^{2}\subset\mathbb{R}^{3}$ is given by
\begin{equation}\label{med}
\vd v(\theta,\varphi)= r^{2}\sin\theta\vd\theta\wedge\vd \varphi=r^{2}\vd\Omega
\end{equation}
where the complex number $v$ is defined via stereographic projection
$$z=x+iy=r\cot \frac{\theta}{2}e^{i\varphi}.$$
and $\vd\Omega$ is the differential solid angle which yields  $\Omega=4\pi$. Thus, integrating \eqref{med} yields
$v=4\pi r^{2}$.

An important remark followed this result which is that $\mathscr{I}_2$ is not $\|\cdot\|-$closed and when $\mathscr{H}=L^{2}(M,\vd\mu)$ (that is, a space of square-integrable functions over a measurable space $(M,\vd\mu)$), $\mathscr{I}_{2}$ is also true.  The next result follows from these remarks.
\begin{teo}[~\cite{RS}, Theorem~VI.23, p.210]\label{HST}Let $(M,\vd\mu )$ be a measurable space and $\mathscr{H}=L^{2}(M,\vd\mu )$. Then $T_{K}\in\mathscr{B}(\mathscr{H})$ is a Hilbert-Schmidt operator if and only if there is a function $K\in L^{2}(M\times M,\vd\mu \otimes\vd\mu) $
with $$T_{K}f(x)=\int K(x,y)f(y)\vd \mu(y)$$ and moreover $$\|T_{K}\|_{2}^{2}=\int |K(x,y)|^{2}\vd\mu(x)\vd\mu(y)$$
where $T_{K}$ is the integral operator associated with $K$.
\end{teo}
\begin{proof}
For proof see Reed and Simons~(\cite{RS}, Theorem~VI.23, pp.210-211).\qedhere
\end{proof}
The work of Ali~\emph{et. al.} on Holomorphic kernels (which in this case are defined on the Riemann sphere) substantiates the square integrability of $T_{K}$ (see~\cite{AAG}, \S6.2, Lemma 6.2.1 and Theorem 6.2.2, along with their proofs, pp.112-116).

With the foregoing in place, it is now appropriate to introduce the spectral shift function using the Kre\u{i}n's trace formula. If $(\mathscr{H}, \langle\cdot,\cdot\rangle)$ is a Hilbert space,$ A \in \mathscr{B}_{1}(\mathscr{H})$ and $N$ is a basis of $\mathscr{H}$, then the trace of the operator $A$ belonging to the trace-class $\mathscr{B}_{1}(\mathscr{H})$ is defined by
$$\mathrm{Tr}_{\mathscr{H}}(A) := \sum_{u\in N}\langle u,Au\rangle$$
(cf:~\cite{VM},Proposition 4.34,\S 4.44,p.191). The trace of the function of an operator $H$ is defined by
\begin{equation}\label{NT}
\mathrm{Tr}_{\mathscr{H}}(f(\mathbf{H})):=\sum_{\varphi\in\mathscr{B}}\langle \varphi,f(\mathbf{H})\varphi\rangle
\end{equation}
where $\mathscr{B}$ is the basis of the Hilbert space $\mathscr{H}=L^{2}(\Omega)$. For any two operators $\mathbf{H}_{k}\;(k=1,2)$
\begin{equation}\label{NT1}
\mathrm{Tr}_{\mathscr{H}}(f(\mathbf{H}_{1})-f(\mathbf{H}_{2})):=\sum_{\varphi\in\mathscr{B}}\langle \varphi,(f(\mathbf{H}_{1})-f(\mathbf{H}_{2}))\varphi\rangle=\int_{\Omega}f'(\lambda)\xi(\lambda;\mathbf{H}_{1},\mathbf{H}_{2})\vd\lambda
\end{equation}
where the function $f\in\mathscr{C}_{c}^{\infty}(\Omega), \Omega\subseteq \mathbb{R}$  is defined as function of an operator in \eqref{rcc1} and $\langle \varphi,\varphi\rangle=\|\varphi\|_{L^{2}(\Omega)}^{2}=1$. The function $\xi(\lambda;\mathbf{H}_{1},\mathbf{H}_{2})$ is called the spectral shift function (SSF). The SSF was first discovered by a great theoretical physicist, I.M. Lifshitz (1917-1982) in 1947. He considered perturbations of an operator $\mathbf{H}_{0}$ arising as Hamiltonian of a lattice model in quantum mechanics (of which Lam\'e operator is a typical example) (see \cite{CGL}:71). A great Mathematician further developed the work of Lifshitz called M.G. Krein (1907-1989). The operator $\mathbf{H}_{0}$ is perturbed by a finite-rank perturbation $V$.

In what follows let $\mathbf{H}, \mathbf{H}_{0}$ be self-adjoint operators in $\mathscr{H}$ with the domains $dom(\mathbf{H}_{0})=dom(\mathbf{H})$ and $V=\mathbf{H}-\mathbf{H}_{0}\in\mathscr{B}_{1}(\mathscr{H})$. Assuming $V\mathcal{R}_{z}(\mathbf{H}_{0})\in\mathscr{B}_{1}(\mathscr{H})$ it is convenient to introduce the perturbation determinant as
\begin{eqnarray}\label{detf}
% \nonumber to remove numbering (before each equation)
\triangle_{\mathbf{H}/\mathbf{H}_{0}}(z) &=& \det{}_{\mathscr{H}}(I+V\mathcal{R}_{z}(\mathbf{H}_{0}))\nonumber \\
&=&\det{}_{\mathscr{H}}(I+(\mathbf{H}-\mathbf{H}_{0})\mathcal{R}_{z}(\mathbf{H}_{0}))\nonumber \\
&=& \det{}_{\mathscr{H}}(\mathbf{H}-zI)(\mathbf{H}_{0}-zI)^{-1}),\;\;\;\;z\in\mathbb{C}.
\end{eqnarray}
The mapping $z\mapsto\triangle_{\mathbf{H}/\mathbf{H}_{0}}(z)$ is analytic in the half-plane $\Im z\in\mathbb{R}\setminus \{0\}$ and thus is an Herglotz function. This function has its complex conjugate given by
$$\triangle_{\mathbf{H}/\mathbf{H}_{0}}(\overline{z})=\overline{\triangle_{\mathbf{H}/\mathbf{H}_{0}}(z)},\;\;\;\Im z\neq 0.$$
Further since $V\in\mathscr{B}_{1}(\mathscr{H})$, standard properties of resolvent imply that $\|V\mathcal{R}_{z}(\mathbf{H})\|_{1}\rightarrow 0$ as $|\Im z|\rightarrow\infty$ and thus $\triangle_{\mathbf{H}/\mathbf{H}_{0}}(z)\rightarrow 1$ as $|\Im z|\rightarrow\infty$.
Since the function $\triangle_{\mathbf{H}/\mathbf{H}_{0}}(z)$ is analytic both in the lower and upper half-planes it is an important fact in complex analysis that there exists an analytic multi-valued function $G(z)$ which is periodic in $2i\pi  k,(i=\sqrt{-1}, k\in\mathbb{Z})$ with $e^{G(z)}=\triangle_{\mathbf{H}/\mathbf{H}_{0}}(z).$ Naturally, it follows that $G(z)=\ln \triangle_{\mathbf{H}/\mathbf{H}_{0}}(z)$ and by fixing the branch of $G(z)$ it is easy to see that $\displaystyle\lim_{|\Im z|\rightarrow\infty}G(z)=0$ .

It is necessary here to take a brief look at the contribution of M.G. Krein to the advancement of the study of SSF.
\begin{teo} [~\cite{MGK}, p.134] Let $\mathbf{H}, \mathbf{H}_{0}$ be self-adjoint operators in $\mathscr{H}$, $\mathbf{H}=\mathbf{H}_{0}+V$ and $V\in\mathscr{B}_{1}(\mathscr{H})$. Then
\begin{equation}\label{KTF1}
\ln\triangle_{\mathbf{H}/\mathbf{H}_{0}}(z)=\int_{\mathbb{R}^{1}}\frac{\xi(\lambda;\mathbf{H},\mathbf{H}_{0})}{\lambda-z}\vd\lambda,\;\;\Im z> 0.
\end{equation}
with $\xi(\lambda;\mathbf{H},\mathbf{H}_{0})\in L^{1}(\mathbb{R}^{1})$ where
\begin{equation}\label{KTF2}
\int_{\mathbb{R}^{1}}|\xi(\lambda;\mathbf{H},\mathbf{H}_{0})|\vd\lambda\leq\|V\|_{1}\;\;\;\mathrm{and}\;\;\;\int_{\mathbb{R}^{1}}\xi(\lambda;\mathbf{H},\mathbf{H}_{0})\vd\lambda=\mathrm{Tr}_{\mathscr{H}}(V)
\end{equation}
The function $\xi(\lambda;\mathbf{H},\mathbf{H}_{0})$ is uniquely determined by Steiltjes inversion formula
\begin{equation}\label{KTF3}
\xi(\lambda;\mathbf{H},\mathbf{H}_{0})=\frac{1}{\pi}\lim_{\epsilon\downarrow 0}\arg \triangle_{\mathbf{H}/\mathbf{H}_{0}}(\lambda+i\epsilon),\;\;\; i=\sqrt{-1}.
\end{equation}
Here $L^{1}(\mathbb{R}^{1})$ denotes the space of all Lesbegue integrable functions on $\mathbb{R}^{1}$ and \eqref{KTF3} is the Privalov's representation of the SSF.
\end{teo}
\begin{proof}
	By the expressions in \eqref{trx} and \eqref{detf}, we have
	\begin{eqnarray}
	% \nonumber to remove numbering (before each equation)
	G(z) &=&\ln\triangle_{\mathbf{H}/\mathbf{H}_{0}}(z) \nonumber \\
	&=&\ln \det{}_{\mathscr{H}}(\mathbf{H}-zI)(\mathbf{H}_{0}-zI)^{-1})\nonumber\\
	&=& \mathrm{Tr}_{\mathscr{H}}(\ln (\mathbf{H}-zI)(\mathbf{H}_{0}-zI)^{-1}))\nonumber \\
	&=& \mathrm{Tr}_{\mathscr{H}}(\ln (\mathbf{H}-zI)-\ln(\mathbf{H}_{0}-zI)).
	\end{eqnarray}
	By setting $f(\mathbf{H})=\ln (\mathbf{H}-zI)$ the corresponding characteristic polynomial $f(\lambda)=\ln(\lambda-z)$ and its first derivative $f'(\lambda)=(\lambda-z)^{-1}$. If this is substituted in \eqref{NT1} above we obtain $G(z)$ given as in the expression~\eqref{KTF1}.
	Next, suppose that the $\mathrm{rank}(V)$ is finite, that is, say $n$. We write
	\begin{equation}\label{KTF4}
	V=\sum_{k=1}^{n}\gamma_{k}\langle \cdot,\varphi_{k}\rangle\varphi_{k}, \;\;\;\;\gamma_{k}=\overline{\gamma_{k}},\;\;\;\;\|\varphi_{k}\|=1\;\;(1\leq k\leq n).
	\end{equation}
	Let  $V_{m}=:\sum_{k=1}^{m}\gamma_{k}\langle \cdot,\varphi_{k}\rangle\varphi_{k}, \;\;\;\;\mathbf{H}_{0}+V_{m}\;\;\;\;\|\varphi_{k}\|=1\;\;(1\leq m\leq \mathrm{rank}(V).$
	Then $\mathbf{H}_{m}-\mathbf{H}_{m-1}=V_{m}-V_{m-1}$ is a rank one operator, that is, $\mathrm{rank}(V_{m}-V_{m-1})=1$ and thus, $V_{m}-V_{m-1}=\gamma_{m}\langle \cdot,\varphi_{m}\rangle\varphi_{m}$ . Let $\mathbf{H}_{n}=H$ then
	\begin{equation}\label{KTF5}
	G(z)=\ln\triangle_{\mathbf{H}/\mathbf{H}_{0}}(z)=\ln\left(\prod_{m=1}^{n}\triangle_{\mathbf{H}_{m}/\mathbf{H}_{m-1}}(z)\right)=\sum_{m=1}^{n}\ln\triangle_{\mathbf{H}_{m}/\mathbf{H}_{m-1}}(z).
	\end{equation}
	Now by expression~\eqref{KTF4}, we define the determinant function
	\begin{eqnarray}\label{KTF6}
	\triangle{}_{\mathscr{H}}(z)&=&\det{}_{\mathscr{H}}(I+V\mathcal{R}_{z}(H_{0}))\nonumber\\
	&=& \det{}_{\mathscr{H}}I+\det{}_{\mathscr{H}}V\mathcal{R}_{z}(H_{0})\nonumber\\
	&=& 1 + \gamma\langle \mathcal{R}_{z}(H_{0})\varphi,\varphi\rangle\nonumber\\
	&=& 1 + \gamma \int_{0}^{\infty}(\lambda-z)^{-1}\vd\langle P_{0}(\lambda)\varphi,\varphi\rangle,\;\;\;\gamma\in\mathbb{R}^{1}
	\end{eqnarray}
	(cf:~\cite{YD}:70). Here, $P_{0}(\lambda)$ is a projection valued measure. Let $z=\lambda+i\epsilon$, then, equation~\eqref{KTF6} becomes
	\begin{eqnarray}
	% \nonumber to remove numbering (before each equation)
	\triangle{}_{\mathscr{H}}(\lambda+i\epsilon)  &=& 1 - \frac{\gamma}{\epsilon} \int_{-\infty}^{\infty}\vd\langle P_{0}(\lambda)\varphi,\varphi\rangle \nonumber\\
	&=&  1 - \frac{\gamma}{\epsilon}\langle\varphi,\varphi\rangle \int_{-\infty}^{\infty} P_{0}(\vd\lambda) \nonumber\\
	&=& 1 - \frac{\gamma}{i\epsilon} \langle \varphi,\varphi\rangle\nonumber \\
	&=& 1 - \frac{\gamma}{i\epsilon}\|\varphi\|_{L^{2}(\mathbb{R}^{1})}^{2}\nonumber\\
	&=& 1 +i \frac{\gamma}{\epsilon},\;\;\;\;\;\;\;\;\;\;i=\sqrt{-1}\label{KTF7}
	\end{eqnarray}
	so that $\triangle{}_{\mathscr{H}}(\lambda+i\epsilon)\rightarrow 1$ as $\epsilon\rightarrow \infty.$
	We note also that the expression in \eqref{KTF6} through \eqref{KTF7} can also be written in the form
	\begin{equation}\label{KTF8}
	\triangle_{\mathbf{H}/\mathbf{H}_{0}}(z)  = 1 - \int_{\mathbb{R}^{1}}(\lambda-z)^{-1}\vd_{\varphi}\mu(\lambda)=1+i\int_{\mathbb{R}^{1}}(\lambda-z)^{-2}\Im z\vd_{\varphi}\mu(\lambda),
	\end{equation}
	where $\Im z$ is the imaginary part of $z\in\mathbb{C}$ and $\vd_{\varphi}\mu(\cdot)$ is a Borel measure on $\mathbb{R}^{1}$. This shows that
	$$\Im\triangle{}_{\mathbf{H}/\mathbf{H}_{0}}(z)  = \int_{\mathbb{R}^{1}}(\lambda-z)^{-2}\Im z\vd_{\varphi}\mu(\lambda)$$
	and thus
	$$\frac{\Im\triangle{}_{\mathbf{H}/\mathbf{H}_{0}}(z)}{\Im z}  = \int_{\mathbb{R}^{1}}(\lambda-z)^{-2}\vd_{\varphi}\mu(\lambda)>0.$$
	Choosing from the branch of the function
	$$\ln \triangle_{\mathbf{H}/\mathbf{H}_{0}}(z)=\ln |\triangle_{\mathbf{H}/\mathbf{H}_{0}}(z)|+i\arg\triangle_{\mathbf{H}/\mathbf{H}_{0}}(z)$$
	with $0<\arg\triangle_{\mathbf{H}/\mathbf{H}_{0}}(z)<\pi$ it is also clear that $0<\frac{1}{\pi}\arg\triangle_{\mathbf{H}/\mathbf{H}_{0}}(z)<1$.
	By taking the natural logarithm of both sides of \eqref{KTF6}, one obtains
	\begin{equation}\label{KTF9}
	\ln\triangle_{\mathbf{H}/\mathbf{H}_{0}}(\lambda+i\epsilon)=\ln\left(1+i\frac{\gamma}{\epsilon}\right).
	\end{equation}
	We recall that
	\[\ln(1+y)=\int_{0}^{y}\frac{\vd t}{1+t}=\int_{0}^{y}\sum_{k=0}^{\infty}(-1)^{k}t^{k}\vd t=\sum_{k=0}^{\infty}(-1)^{k}\frac{t^{k+1}}{k+1}\bigg|_{0}^{y}\simeq y+o(y^{2})\]
	as $y\rightarrow 0$.
	Thus, expression~\eqref{KTF8} becomes
	\begin{equation}\label{KTF9b}
	\ln\triangle_{\mathbf{H}/\mathbf{H}_{0}}(\lambda+i\epsilon)\simeq i\frac{\gamma}{\epsilon}+o(\left(i\frac{\gamma}{\epsilon}\right)^{2}).
	\end{equation}
	The isolated part of $\ln\triangle_{\mathbf{H}/\mathbf{H}_{0}}(\lambda+i\epsilon)$ the expression~\eqref{KTF9} is the imaginary part as $\ln|\triangle_{\mathbf{H}/\mathbf{H}_{0}}(\lambda+i\epsilon)|\rightarrow 0$ as $i\frac{\gamma}{\epsilon}\rightarrow 0$.
	Therefore
	\begin{equation}\label{KTF10}
	0<\frac{1}{\pi}\arg\triangle_{\mathbf{H}/\mathbf{H}_{0}}(\lambda+i\epsilon)=\frac{1}{\pi}\Im\ln\triangle_{\mathbf{H}/\mathbf{H}_{0}}(\lambda+i\epsilon)=\frac{\gamma}{\epsilon}<1.
	\end{equation}
	Next, to obtain the Privalov's representation of SSF \eqref{KTF3}, we state the Steiltjes-Perron transform
	\begin{equation}\label{STF}
	\Phi(z)= \mathcal{S}_{z}[F]:=\int_{0}^{\infty}\frac{F(y)}{y+z}\vd y
	\end{equation}
	with its inverse transform
	\begin{equation}\label{STF2}
	F(y)=\lim_{\epsilon\rightarrow 0^{+}}\frac{1}{2\pi i}[\Phi(-y-i\epsilon)-\Phi(-y+i\epsilon)],\;\;\;\; y\in \Omega \subset \mathbb{R}^{1}, z\in\mathbb{C}\setminus\mathbb{R}
	\end{equation}
	(cf:~\cite{HSW} , p.247). Now let $$G(z)=\int_{\mathbb{R}^{1}}\frac{\xi(\lambda;\mathbf{H},\mathbf{H}_{0})\vd\lambda}{\lambda-z},\;\;\;\Im z=\epsilon >0.$$
	It follows by \eqref{STF} and \eqref{STF2} that
	\begin{eqnarray}\label{KTF11}
	% \nonumber to remove numbering (before each equation)
	\xi(\lambda;\mathbf{H},\mathbf{H}_{0})&=&\lim_{\epsilon\rightarrow 0^{+}}\frac{1}{2\pi i}[G(\lambda+i\epsilon)-G(\lambda-i\epsilon)].\nonumber \\
	&=& \lim_{\epsilon\rightarrow 0^{+}}\frac{1}{2\pi i}\cdotp \frac{2i}{\epsilon}\int_{\mathbb{R}^{1}}\xi(\lambda;\mathbf{H},\mathbf{H}_{0})\vd\lambda \nonumber\\
	&=&\frac{1}{\pi}\lim_{\epsilon\rightarrow 0^{+}}\int_{\mathbb{R}^{1}}\frac{\xi(\lambda;\mathbf{H},\mathbf{H}_{0})}{\epsilon}\vd\lambda\nonumber\\
	&=&\frac{1}{\pi}\lim_{\epsilon\downarrow 0}\arg\triangle_{\mathbf{H}/\mathbf{H}_{0}}(\lambda+i\epsilon).
	\end{eqnarray}
	By \eqref{KTF10} and \eqref{KTF11} it is obvious that $0<\xi(\lambda;\mathbf{H},\mathbf{H}_{0})<1$. It is well-known that if the operator $V=\mathbf{H}-\mathbf{H}_{0}$ has precisely $p$ positive eigen-values ($\nu$ negative eigen-values) then $\xi(\lambda;\mathbf{H},\mathbf{H}_{0})\leq p$ ($\xi(\lambda;\mathbf{H},\mathbf{H}_{0})\geq -\nu)$). This implies by the above inequality that $\nu=0$ and $p$. Let us suppose that $V$ is arbitrary trace class perturbation and $V_{n}$ be a sequence of finite rank operators, such that $\|V-V_{n}\|_{1}\rightarrow 0$ as $n\rightarrow \infty$. Then setting $$\xi(\lambda;\mathbf{H},\mathbf{H}_{0})=\sum_{m=1}^{\infty}\xi(\lambda;\mathbf{H}_{m},\mathbf{H}_{m-1}).$$
	The sum is finite only if rank$(V)=n$. Let this be the case, then $$\xi(\lambda;\mathbf{H},\mathbf{H}_{0})=\sum_{m=1}^{n}\xi(\lambda;\mathbf{H}_{m},\mathbf{H}_{m-1})$$
	and $|\xi(\lambda;\mathbf{H}_{m},\mathbf{H}_{m-1})|=|\gamma_{m}|\left(=\pm \mathrm{Tr}(\mathbf{H}_{m}-\mathbf{H}_{m-1})\right)$. Consequent to this equality, the series $\sum_{m=1}^{\infty}\xi(\lambda;\mathbf{H}_{m},\mathbf{H}_{m-1})$ converges absolutely in the metric $L^{1}(\mathbb{R})$ to some function $\xi(\lambda;\mathbf{H},\mathbf{H}_{0})$ and $$\lim_{n\rightarrow \infty}\int_{\mathbb{R}}\frac{\sum_{m=1}^{n}\xi(\lambda;\mathbf{H}_{m},\mathbf{H}_{m-1})}{\lambda-z}\vd\lambda=\int_{\mathbb{R}}\frac{\xi(\lambda;\mathbf{H},\mathbf{H}_{0})}{\lambda-z}\vd\lambda.$$
	Let $V_{n}=:\sum_{m=1}^{n}\gamma_{m}\langle\cdot,\varphi_{m}\rangle\varphi_{m}$. Then $$\triangle_{\mathbf{H}_{n}/\mathbf{H}_{0}}(z)=\det[I+V_{n}\mathcal{R}_{z}(\mathbf{H}_{0})].$$
	By strong convergence $$\|V-V_{n}\|_{1}=\|\sum_{m=n+1}^{\infty}\gamma_{m}\langle\cdot,\varphi_{m}\rangle\varphi_{m}\|_{1}=\sum_{m=n+1}^{\infty}|\gamma_{m}|\rightarrow 0$$
	as $n\rightarrow\infty$ and thus
	$$\triangle_{\mathbf{H}_{n}/\mathbf{H}_{0}}(z)\rightarrow\det[I+V\mathcal{R}_{z}(\mathbf{H}_{0})]=\triangle_{\mathbf{H}/\mathbf{H}_{0}}(z).$$
	It is therefore evident that
	$$\ln\triangle_{\mathbf{H}/\mathbf{H}_{0}}(z)=\int_{\mathbb{R}}\frac{\xi(\lambda;\mathbf{H},\mathbf{H}_{0})}{\lambda-z}\vd\lambda.$$
	Now, the $L^{1}-$norm of the function $\xi(\lambda)$
	\begin{eqnarray*}
		% \nonumber to remove numbering (before each equation)
		\|\xi(\lambda;\mathbf{H},\mathbf{H}_{0})\|_{1} &=& \int_{\mathbb{R}}|\xi(\lambda;\mathbf{H},\mathbf{H}_{0})|\vd\lambda \\
		&=& \lim_{n\rightarrow\infty}\int_{\mathbb{R}}|\sum_{m=1}^{n}\xi(\lambda;\mathbf{H}_{m},\mathbf{H}_{m-1})|\vd\lambda \\
		&\leq& \lim_{n\rightarrow\infty}\sum_{m=1}^{n}\int_{\mathbb{R}}|\xi(\lambda;\mathbf{H}_{m},\mathbf{H}_{m-1})|\vd\lambda  \\
		&=& \lim_{n\rightarrow\infty}|\gamma_{m}| \\
		&=& \|V\|_{1}.
	\end{eqnarray*}
	On the other hand,
	\begin{eqnarray*}
		% \nonumber to remove numbering (before each equation)
		\int_{\mathbb{R}^{1}}\xi(\lambda;\mathbf{H},\mathbf{H}_{0})\vd\lambda  &=& \int_{\mathbb{R}^{1}}\sum_{m=1}^{\infty}\xi(\lambda;\mathbf{H}_{m},\mathbf{H}_{m-1})\vd\lambda\\
		&=& \lim_{n\rightarrow\infty}\int_{\mathbb{R}^{1}}\sum_{m=1}^{n}\xi(\lambda;\mathbf{H}_{m},\mathbf{H}_{m-1})\vd\lambda \\
		&=& \sum_{m=1}^{\infty}\gamma_{m}\\
		&=& \mathrm{Tr}_{\mathscr{H}}(V).
	\end{eqnarray*}
	This completes the proof.\qedhere
\end{proof}
We remark here that $\mathscr{B}_{1}(\mathscr{H})\subsetneq \mathscr{B}_{2}(\mathscr{H}) $ and this was illustrated by a counter-example of Hilbert-Schmidt operators that are not trace-class operators due to Kre\u{i}n~\cite{MGK}. A typical example is found in the work of Sinha and Mohapatra (~\cite{SM}, p.820).
\section{Main Results }\label{sec5}
The Brioschi-Halphen equation (BHE) given in \eqref{bh} is a Fuchsian type equation obtained by a two-step transformation of the Lam\'e equation (~\cite{POO}, Chapter IX, \S37, Eq.(16), p.163). In this section, SSF associated with Lam\'e equation and Brioschi-Halphen equation are computed.  Our technique of solution is similar to that explained by Kanwal (see~\cite{KRP},  p.246) in connection with the fundamental solution of a second-order ordinary differential equation with polynomial coefficients. The difference between what follows and Kanwal's work is that our differential equation under consideration is defined in the complex domain and has coefficients that are elliptic functions. The Green function obtained by this procedure is used to calculate the spectral shift function associated with the differential operator of BHE.

\begin{teo}Consider the Lam\'e equation
\begin{equation}\label{Lam}
-y''(z)+2\wp(z)y(z)=By(z)
\end{equation}
satisfied by the function
$$f(z)=K_1\frac{\sigma(z+\varepsilon)}{\sigma(z)}\exp(-z\zeta(\varepsilon))+K_2\frac{\sigma(z-\varepsilon)}{\sigma(z)}\exp(z\zeta(\varepsilon)),$$
where $\sigma(\cdot)$ and $\zeta(\cdot)$ are Weierstrass elliptic sigma and zeta functions respectively and $K_1, K_2$ are arbitrary constants.
Then Green kernel which is associated with the equation~\eqref{Lam}
is given by
\begin{equation*}
\xi(\lambda,w)=\left\{\begin{array}{cc}
\infty & \text{if}\; w\in\mathbb{C}_{+} \\
&  \\
\frac{2\kappa\delta(w)+f(w)\delta'(z-w)}{4\pi}\lim_{\epsilon\downarrow 0} \tan^{-1}\left(-\frac{\epsilon}{\lambda}\right)&  \text{if}\; w\in\mathbb{C}_{-}.
\end{array}\right.
\end{equation*}
where $\kappa$ and $B=\wp(\varepsilon)$ are constants.
\end{teo}
\begin{proof}
	To prove this result, we need to show that $G:=G(z,w)$ is a fundamental solution of the equation
	\begin{equation}\label{FE}
	-\frac{\partial^{2} G}{\partial z^{2}}+2\wp(z)G=\kappa\delta(z-w).
	\end{equation}
	To this end, let the solution of equation~\eqref{FE} be written in the form
	\begin{equation}\label{Gren}
	G(z,w)=f(z)\mathbb{H}(z-w)
	\end{equation}
	where $f(z)$ is an elliptic function satisfying the equation~\eqref{Lam}. Differentiating the expression~\eqref{Gren} twice with respect to $z$ yields
	$$G_{zz}=f''(z)\mathbb{H}(z-w)+f'(w)\delta(z-w)+f(w)\delta'(z-w),$$
	so that
	\begin{equation}\label{Gr}
	-f''(z)\mathbb{H}(z-w)-f'(w)\delta(z-w)-f(w)\delta'(z-w)+2\wp(z)f(z)\mathbb{H}(z-w)=\kappa\delta(z-w).
	\end{equation}
	Since $-f''(z)+2\wp(z)f(z)=\wp(\varepsilon)f(z)$, it follows from equation~\eqref{Gr} that
	$$ -f'(w)\delta(z-w)-f(w)\delta'(z-w)+\wp(\varepsilon)f(z)\mathbb{H}(z-w)=\kappa\delta(z-w),$$
	which may be re-written as
	$$-f'(w)\delta(z-w)+f'(w)\delta(z-w)-f(z)\delta'(z-w)+\wp(\varepsilon)f(z)\mathbb{H}(z-w)=\kappa\delta(z-w).$$
	The last equation gives
	$$G(z,w)=f(z)\mathbb{H}(z-w)=\frac{\kappa\delta(z-w)+f(z)\delta'(z-w)}{\wp(\varepsilon)}.$$
	For $w$ in the positive upper-half plane $(\mathbb{C}_{+})$ or the negative lower half-plane $(\mathbb{C}_{-})$ of the complex plane $\mathbb{C}$, we have
	\begin{equation}\label{Grf}
	G(z,w)=\left\{\begin{array}{cc}
	\frac{\kappa\delta(z-w)+f(z)\delta'(z-w)}{\wp(\varepsilon)} & \text{if}\; w\in\mathbb{C}_{+} \\
	&  \\
	\frac{\kappa\delta(z+w)+f(z)\delta'(z+w)}{\wp(\varepsilon)}&  \text{if}\; w\in\mathbb{C}_{-}.
	\end{array}\right.
	\end{equation}

In what follows, we take the appearance of $\wp(\varepsilon)$ on the righthand side of equation~\eqref{Grf} into consideration, and so write $G(\wp(\varepsilon); z,w)\equiv G(z,w)$. Let $z=w$, then,  $G(\wp(\varepsilon); w,w)=G(\wp(\varepsilon);w)$. Thus, following the Green function given in expression~\eqref{Grf}, we get
\begin{equation}\label{Grf1}
G(\wp(\varepsilon);w)=\left\{\begin{array}{cc}
\frac{\kappa\delta(0)+f(w)\delta'(0)}{\wp(\varepsilon)} & \text{if}\; w\in\mathbb{C}_{+} \\
&  \\
\frac{\kappa\delta(2w)+f(w)\delta'(2w)}{\wp(\varepsilon)}&  \text{if}\; w\in\mathbb{C}_{-}.
\end{array}\right.
\end{equation}
By general principles (\cite{GES}, Eq.(4.8), p.760), $\lim_{\epsilon\downarrow 0}G(\lambda+i\epsilon;w)$ exists for a.e. $\lambda\in \mathbb{R}$, the spectral shift function
$$\xi(\lambda,u)=\frac{1}{\pi}\mathrm{Arg}\left(\lim_{\epsilon\downarrow 0}G(\lambda+i\epsilon;u)\right).$$
Explicitly, this is obtained as
\begin{equation}\label{Grf1b}
\xi(\lambda,w)=\left\{\begin{array}{cc}
\frac{\kappa\delta(0)+f(w)\delta'(0)}{\pi}\lim_{\epsilon\downarrow 0} \tan^{-1}\left(-\frac{\epsilon}{\lambda}\right)& \text{if}\; w\in\mathbb{C}_{+} \\
&  \\
\frac{\kappa\delta(2w)+f(w)\delta'(2w)}{\pi}\lim_{\epsilon\downarrow 0} \tan^{-1}\left(-\frac{\epsilon}{\lambda}\right)&  \text{if}\; w\in\mathbb{C}_{-}.
\end{array}\right.
\end{equation}
Here the complex number $\wp(\varepsilon)=\lambda+i\epsilon, \mathrm{Arg}(\cdot)$ is the complex argument and the limit
$\lim_{\epsilon\downarrow 0} \tan^{-1}\left(-\frac{\epsilon}{\lambda}\right)=\pi$ or $2\pi$ in both cases. Thus, since $\delta(0)=\infty$ and $\delta(2w)=\frac{1}{2}\delta(w)$ and by adapting Kanwal's approach (\cite{KRP},\S3.1, Eq.(7), p.50) for any $w$ in open disc $D(z,r)$ with radius $r$ and center $z$,
\begin{equation}\label{grr}
  \delta'(g(z))=\frac{1}{|g'(w)|^{2}}\left\{\delta'(z-w)+\frac{g''(w)}{g'(w)}\delta(z-w)\right\}
\end{equation}
so that $\delta'(2w)=\frac{1}{4}\delta'(z-w)$ then the expression~\eqref{Grf1} simplifies as
\begin{equation}\label{Grf2}
\xi(\lambda,w)=\left\{\begin{array}{cc}
\infty & \text{if}\; w\in\mathbb{C}_{+} \\
&  \\
\frac{2\kappa\delta(w)+f(w)\delta'(z-w)}{4\pi}\lim_{\epsilon\downarrow 0} \tan^{-1}\left(-\frac{\epsilon}{\lambda}\right)&  \text{if}\; w\in\mathbb{C}_{-}.
\end{array}\right.
\end{equation}
Here $$\lim_{\epsilon\downarrow 0} \tan^{-1}\left(-\frac{\epsilon}{\lambda}\right)=\pi\mathbb{H}(\lambda)$$
is an Heisenberg distribution (see~\cite{KRP}, \S2.7, Example 4, p.45).
We also remark here that since $f(z)$ satisfies
equation~\eqref{Lam} it must be the well-known solution of Lam\'e equation given explicitly by
$$f(z)=K_1\frac{\sigma(z+\varepsilon)}{\sigma(z)}\exp(-z\zeta(\varepsilon))+K_2\frac{\sigma(z-\varepsilon)}{\sigma(z)}\exp(z\zeta(\varepsilon)),$$
where $\sigma(\cdot)$ and $\zeta(\cdot)$ are Weierstrass elliptic sigma and zeta functions respectively.\qedhere
\end{proof}

The next theorem shows the spectral shift function associated with BHE.

\begin{teo} Considering the operator in \eqref{bh} without the accessory parameter in the form
\begin{equation}\label{bh25a}
H_{s}=(4w^3-g_2w-g_3)D^2-(2s-1)(6w^2-\frac{1}{2}g_2)D+4s(2s-1)w,\;\;(D=\frac{\vd}{\vd w}),
\end{equation}
we construct a distributional problem with fundamental solution $\Xi_{p,s}$
\begin{equation}\label{bh25b}
	H_{s}\Xi_{p,s}=\delta(w).
\end{equation}
whose SSF is given as
\[\xi^{\pm}(\lambda,w)= G^{\pm}(w)\mathbb{H}(\lambda)\]
where,
\[G^{\pm}(w) = \frac{i}{a_{+}-a_{-}}\left(e^{i\pm a_{+}w}-e^{\pm i a_{-}w}\right)\Gamma_{p}(\pm w)+\frac{A[e^{\pm i a_{+}w}+e^{\pm ia_{-}w}]}{2\pi\sqrt{X_{1}^{2}+4X_{2}X_{0}}}
\] and $a_{\pm}\in \mathbb{C}_{\pm}$ (upper and lower complex half-plains) are determined.
\end{teo}
\begin{proof}
The Fourier transform is very essential in the resolution of Green functions of differential equations. Let the Fourier transform in this case be defined as follows. Consider the Hilbert space $\mathfrak{H}=L^2(\Gamma,\vd\mu_{\omega}(z)), \Gamma\subset\mathbb{C}$ and its Fourier transformed space $\widetilde{\mathfrak{H}}=L^2(\widetilde{\Gamma},\vd\mu_{\omega}(w)), \widetilde{\Gamma}\subset\widehat{\mathbb{C}}$ where $\vd\mu_\omega(z)=\omega(z)\vd z$ is a weighted measure with weight $$\omega(z)=\prod_{k=1}^{3}(z-e_k)^{\sigma_k-1}=\sum_{m=0}^{\sigma_1-1}\sum_{n=0}^{\sigma_2-1}\sum_{q=0}^{\sigma_3-1}h_{mnq}z^{m+n+q-3}$$ associated with $H_s$.
Here, $$h_{mnq}=\binom{\sigma_1-1}{m}\binom{\sigma_2-1}{n}\binom{\sigma_3-1}{q}e_{1}^{m}e_2^{n}e_3^{q}.$$
The Fourier transform (isometry) $\mathscr{F}:\mathfrak{H}\rightarrow\widetilde{\mathfrak{H}}$ is defined by
$$F(z)=\mathscr{F}[f(w)]=\frac{1}{2\pi}\int_{\Gamma}f(z)e^{-i w\cdot z}\vd\mu_{\omega}(z)$$
(Some references use $izw$ instead of $-i w\cdot z ,\;i=\sqrt{-1},z\in\mathbb{C}, \widehat{\mathbb{C}}$ is the dual of $\mathbb{C}$). Here, $\widehat{\mathbb{C}}$ is defined as
$$\widehat{\mathbb{C}}:=\{1,i,\omega, i\omega|\omega^2=0=i^2\omega^2, i\omega\cdot\omega=0=\omega\cdot i\omega\}.$$
Let $l\equiv l_{mnq}=m+n+q, \partial_{z}=\frac{\vd}{\vd z}$. Applying the Fourier transform $\mathscr{F}$ to equation~\eqref{bh25b} and thus obtain
\begin{equation}\label{bh26}
	\mathscr{F}[H_s]\widehat{\Xi_{p,s}}=1
\end{equation}
where
\begin{multline}\label{bh27a}
\mathscr{F}[H_s]=\sum_{m=0}^{\sigma_1-1}\sum_{n=0}^{\sigma_2-1}\sum_{q=0}^{\sigma_3-1}h_{mnq}\bigg\{	\big[(-i)^{l+5}4\partial_{z}^{(l+3)}+(-i)^{l+3}g_2\partial_{z}^{(l+1)}-(-i)^{l+2}g_3\partial_{z}^{(l)}\big]z^2\\
	-(2s-1)\big[6(-i)^{l+4}\partial_{z}^{(l+2)}-\frac{g_2}{2}(-i)^{l+1}\partial_{z}^{(l)}\big]z-4s(2s-1)(-i)^{l+2}\partial_{z}^{(l+1)}z\bigg\}.
\end{multline}
Equation~\eqref{bh27a} further simplifies as
\begin{multline}\label{bh27b}
\mathscr{F}[H_s]=\sum_{m=0}^{\sigma_1-1}\sum_{n=0}^{\sigma_2-1}\sum_{q=0}^{\sigma_3-1}h_{mnq}\bigg\{	\big[-i(-i)^{l}4\partial_{z}^{(l+3)}+i(-i)^{l}g_2\partial_{z}^{(l+1)}+(-i)^{l}g_3\partial_{z}^{(l)}\big]z^2\\
	-(2s-1)\big[6(-i)^{l}\partial_{z}^{(l+2)}-\frac{g_2}{2}(-i)^{l+1}\partial_{z}^{(l)}\big]z+4s(2s-1)(-i)^{l}\partial_{z}^{(l+1)}z\bigg\}.
\end{multline}
 For non-trivial value of $\mathscr{F}[H_s],$ we set $l\in\{0,1,2\}$.
  \begin{table}
    \centering
    \begin{tabular}{ccc}
     \Xhline{2.5\arrayrulewidth}
          % after \\: \hline or \cline{col1-col2} \cline{col3-col4} ...
     $mnq$ & $h_{mnq}$ & $\psi_{mnq}^{s}$ \\
     \Xhline{2.5\arrayrulewidth}
     $000$ & 1 & $g_{3}z^2+\frac{1}{2}(3-2s)ig_{2}z+4s(2s-1)$ \\

     $100$ & $(\sigma_{1}-1)e_{1}$ & $\frac{1}{2}(5-2s)g_{2}-2ig_{3}z$ \\

     $010$ & $(\sigma_{2}-1)e_{2}$  & $\frac{1}{2}(5-2s)g_{2}-2ig_{3}z$ \\

     $001$ & $(\sigma_{3}-1)e_{3}$ & $\frac{1}{2}(5-2s)g_{2}-2ig_{3}z$ \\

     $110$& $(\sigma_{1}-1)(\sigma_{2}-1)e_{1}e_{2}$&$-2g_{3}$\\

     $101$& $(\sigma_{1}-1)(\sigma_{3}-1)e_{1}e_{3}$&$-2g_{3}$\\

     $011$&$(\sigma_{2}-1)(\sigma_{3}-1)e_{2}e_{3}$&$-2g_{3}$\\

     $200$ & $\frac{(\sigma_{1}-1)(\sigma_1-2)}{2}e_1^{2}$ & $-2g_{3}$ \\

     $020$ & $\frac{(\sigma_{2}-1)(\sigma_{2}-2)}{2}e_2^{2}$ & $-2g_{3}$ \\

     $002$ & $\frac{(\sigma_{3}-1)(\sigma_{3}-2)}{2}e_3^{2}$ & $-2g_{3}$ \\
    \Xhline{2.5\arrayrulewidth}
   \end{tabular}
    \caption{Entries of $\mathscr{F}[H_{s}]$}\label{tabhc}
  \end{table}
  Let $\mathscr{F}[H_s]$ be written in the form
 $$\mathscr{F}[H_s]=\sum_{m=0}^{\sigma_1-1}\sum_{n=0}^{\sigma_2-1}\sum_{q=0}^{\sigma_3-1}h_{mnq}\psi_{mnq}^{s}(z,\partial_{z}).$$
 Then the simplify expression of $\mathscr{F}[H_s]$ is given by
 \begin{multline}\label{fhs}
   \mathscr{F}[H_s]=h_{000}\psi_{000}^{s}+\left(h_{100}+h_{010}+h_{001}\right)\psi_{100}^{s} \\
   +\left(h_{110}+h_{101}+h_{011}+h_{200}+h_{020}+h_{002}\right)\psi_{200}^{s} .
 \end{multline}
 Explicitly using the entries in Table~\ref{tabhc}, equation~\eqref{fhs} becomes
 \begin{multline}\label{fhs1}
    \mathscr{F}[H_s]=g_{3}z^{2}+i(3-2s)\frac{g_{3}}{2}z+4s(2s-1)+(\sigma_{1}e_{1}+\sigma_{2}e_{2}+\sigma_{3}e_{3})\big((5-2s)\frac{g_{2}}{2}-2ig_{3}z\big)\\
   - 2\bigg((\sigma_{1}-1)(\sigma_{2}-1)e_{1}e_{2}+(\sigma_{2}-1)(\sigma_{3}-1)e_{2}e_{3}+(\sigma_{1}-1)(\sigma_{3}-1)e_{1}e_{3}\\
   +\frac{(\sigma_{1}-1)(\sigma_1-2)}{2}e_{1}^{2}+ \frac{(\sigma_{2}-1)(\sigma_{2}-2)}{2}e_2^{2}+\frac{(\sigma_{3}-1)(\sigma_{3}-2)}{2}e_3^{2}\bigg)g_{3}.
 \end{multline}
 Concisely,
 $$\mathscr{F}[H_{s}]=X_{2}z^{2}+iX_{1}z+X_{0}$$
 where
 $$X_{2}=g_{3}, X_{1}=\left[(3-2s)\frac{g_{3}}{2}-2(\sigma_{1}e_{1}+\sigma_{2}e_{2}+\sigma_{3}e_{3})g_{3}\right],$$
 \begin{multline}\label{fhs2}
 X_{0}=(\sigma_{1}e_{1}+\sigma_{2}e_{2}+\sigma_{3}e_{3})(5-2s)\frac{g_{2}}{2}
   - 2\bigg((\sigma_{1}-1)(\sigma_{2}-1)e_{1}e_{2}\\ +(\sigma_{2}-1)(\sigma_{3}-1)e_{2}e_{3}+(\sigma_{1}-1)(\sigma_{3}-1)e_{1}e_{3}
   +\frac{(\sigma_{1}-1)(\sigma_1-2)}{2}e_{1}^{2}\\+ \frac{(\sigma_{2}-1)(\sigma_{2}-2)}{2}e_2^{2}+\frac{(\sigma_{3}-1)(\sigma_{3}-2)}{2}e_3^{2}\bigg)g_{3}.
 \end{multline}
 Thus,
 \begin{equation}\label{bh29a}
   \widehat{\Xi_{p,s}}=\frac{1}{X_{2}z^{2}+X_{1}z+X_{0}}=\frac{1}{(z-a_{+})(z-a_{-})}\equiv\frac{1}{a_{+}-a_{-}}\left(\frac{1}{z-a_{+}}-\frac{1}{z-a_{-}}\right).
 \end{equation}
  where
  $$a_{\pm}=\frac{-iX_{1}\pm\sqrt{-X_{1}^{2}-4X_{2}X_{0}}}{2X_{2}}=\frac{i[-X_{1}\pm\sqrt{X_{1}^{2}+4X_{2}X_{0}}]}{2X_{2}}.$$
  In this case, the discriminant  $X_{1}^{2}+4X_{2}X_{0}>0$, thus $a_{\pm}$ is purely imaginary. By linearity of $\mathscr{F}^{-1}$
\begin{equation}\label{fhs3}
  \Xi_{p,s}=\frac{1}{a_{+}-a_{-}}\left(\mathscr{F}^{-1}\bigg[\frac{1}{z-a_{+}}\bigg]-\mathscr{F}^{-1}\bigg[\frac{1}{z-a_{-}}\bigg]\right).
\end{equation}
Considering Signum function, $\mathrm{sgn}(t)=\left\{\begin{array}{ccc}
                                                       1, & \textrm{if}& t>0 \\
                                                       0, & \textrm{if} &t=0 \\
                                                       -1, & \textrm{if} & t<0
                                                     \end{array}\right.
$ and step function \[u(t)=\frac{1}{2}+\frac{1}{2}\mathrm{sgn}(t)=\left\{\begin{array}{ccc}
                                                                          1, & \textrm{if} & t>0 \\
                                                                          0, & \textrm{if} & t<0
                                                                        \end{array}\right.
\] (see~\cite{ADP}, Table~1.2, p.1-3), the identity
\begin{equation}\label{fhs4}
  \mathscr{F}^{-1}\left[\frac{1}{(z-a_{\pm})^{m}}\right]\bigg|_{t}=i\frac{(it)^{m-1}}{(m-1)!}e^{ia_{\pm}t}\Gamma_{p}(t), \;\textrm{(cf:~\cite{ADP}, \S2.118, p.2-29)}
\end{equation}
where $p$ is the imaginary part of $a_{\pm}$ and
\begin{equation}\label{fhs5}
  \Gamma_{p}(t)=\left\{\begin{array}{ccc}
                         u(t), & \textrm{ if }& p>0 \\
                         -\frac{1}{2}\textrm{sgn}(t), & \textrm{ if } &p=0 \\
                         -u(-t), & \textrm{ if }& p<0.
                       \end{array}\right.
\end{equation}
Here, the step function $u(t)$ is the same as the Heaviside function $\mathbb{H}(t).$ Applying the identities in \eqref{fhs4} and \eqref{fhs5} to \eqref{fhs3} one obtains
\begin{equation}\label{fhs6}
  \Xi_{p,s}=\frac{i}{a_{+}-a_{-}}\left(e^{ia_{+}w}-e^{ia_{-}w}\right)\Gamma_{p}(w).
\end{equation}
The complementary solution $\Xi_{c,s}$ of equation~\eqref{bh26} satisfies the homogeneous equation
$$\mathscr{F}[H_s]\widehat{\Xi_{c,s}}=0$$
is given by the surface distribution $\Xi_{c,s}=A\delta\left(\mathscr{F}[H_s]\right)$
(cf:~\cite{KRP},~\S6.4, Eq.(80)-Eq.(82),~p.165).
Considering the fact that $a_{\pm}$ are the roots of $\mathscr{F}[H_s]$ (\cite{KRP}, \S3.1, Eq.(2), p.53) it is obvious that
\begin{equation}\label{grr2}
  \delta(\mathscr{F}[H_s])=\frac{\delta(z-a_{+})}{|\partial_{z}\mathscr{F}[H_s]|\big|_{a_{+}}}+\frac{\delta(z-a_{-})}{|\partial_{z}\mathscr{F}[H_s]|\big|_{a_{-}}}.
\end{equation}
 The Green function associated with BHE is
\begin{eqnarray*}
% \nonumber % Remove numbering (before each equation)
  G(z,w)&=&\Xi_{c,s}+\Xi_{p,s}\\
   &=& \frac{i}{a_{+}-a_{-}}\left(e^{ia_{+}w}-e^{ia_{-}w}\right)\Gamma_{p}(t)+A\mathscr{F}^{-1}\bigg\{\frac{\delta(z-a_{+})}{|2X_{2}a_{+}+iX_{1}|}+\frac{\delta(z-a_{-})}{|2X_{2}a_{-}+iX_{1}|}\bigg\} \\   &=&\frac{i}{a_{+}-a_{-}}\left(e^{ia_{+}w}-e^{ia_{-}w}\right)\Gamma_{p}(w)+A\frac{e^{ia_{+}w}}{2\pi}\ast\mathscr{F}^{-1}\bigg\{\frac{1}{|2X_{2}a_{+}+iX_{1}|}\bigg\}\\
   &&+A\frac{e^{ia_{-}w}}{2\pi}\ast\mathscr{F}^{-1}\bigg\{\frac{1}{|2X_{2}a_{-}+iX_{1}|}\bigg\}\\
   &=&\frac{i}{a_{+}-a_{-}}\left(e^{ia_{+}w}-e^{ia_{-}w}\right)\Gamma_{p}(w)+A\frac{e^{ia_{+}w}}{2\pi\sqrt{X_{1}^{2}+4X_{2}X_{0}}}\ast\mathscr{F}^{-1}[1]\\
   &&+A\frac{e^{ia_{-}w}}{2\pi\sqrt{X_{1}^{2}+4X_{2}X_{0}}}\ast\mathscr{F}^{-1}[1]
   \end{eqnarray*}
   \begin{eqnarray*}
   &=&\frac{i}{a_{+}-a_{-}}\left(e^{ia_{+}w}-e^{ia_{-}w}\right)\Gamma_{p}(w)+A\frac{e^{ia_{+}w}+e^{ia_{-}w}}{2\pi\sqrt{X_{1}^{2}+4X_{2}X_{0}}}\ast\mathscr{F}^{-1}[1]\\
   &=&\frac{i}{a_{+}-a_{-}}\left(e^{ia_{+}w}-e^{ia_{-}w}\right)\Gamma_{p}(w)+A\frac{e^{ia_{+}w}+e^{ia_{-}w}}{2\pi\sqrt{X_{1}^{2}+4X_{2}X_{0}}}\ast\delta(w)\\
      &=&\frac{i}{a_{+}-a_{-}}\left(e^{ia_{+}w}-e^{ia_{-}w}\right)\Gamma_{p}(w)\\
      &&\;\;\;\;+\frac{A}{2\pi\sqrt{X_{1}^{2}+4X_{2}X_{0}}} \int_{-i\infty}^{+i\infty}[e^{ia_{+}w}+e^{ia_{-}w}]\delta(w-z)\vd z\\
     &=&\frac{i}{a_{+}-a_{-}}\left(e^{ia_{+}w}-e^{ia_{-}w}\right)\Gamma_{p}(w)+\frac{A[e^{ia_{+}z}+e^{ia_{-}z}]}{2\pi\sqrt{X_{1}^{2}+4X_{2}X_{0}}}.\\
   \end{eqnarray*}
By setting $z=\pm w\in\Omega_{\pm}$ (in the upper and lower complex half-plane) Green function exists and can be given by
\begin{equation}\label{grr4}
  G^{\pm}(w) = \frac{i}{a_{+}-a_{-}}\left(e^{i\pm a_{+}w}-e^{\pm i a_{-}w}\right)\Gamma_{p}(\pm w)+\frac{A[e^{\pm i a_{+}w}+e^{\pm ia_{-}w}]}{2\pi\sqrt{X_{1}^{2}+4X_{2}X_{0}}}.
\end{equation}
Let $4B=\lambda+i\varepsilon\in\mathbb{C}\setminus\{0\}$ be the eigenvalue of Brioschi-Halphen operator $H_{s}$ in terms of the Lam\'e accessory parameter $B$. Then, by Kanwal (cf:~\cite{KRP},\S 2.7, p.45), for any $\varepsilon\in\mathbb{R}^{+},$
$$\lim_{\varepsilon\downarrow 0}\arg(\lambda+i\varepsilon)=\pi\mathbb{H}(-\lambda).$$
It is important to note here that $$\arg(z)=\mathrm{Arg}(z)+2\pi n, n\in\mathbb{Z},$$ where $-\pi<\mathrm{Arg}(z)\leq\pi.$
The SSF following Gesztesy (\cite{GES}, Eq.(4.8), p.760), $\lim_{\epsilon\downarrow 0}G(\lambda+i\varepsilon;w)$ exists for a.e. $\lambda\in \mathbb{R}$, the spectral shift function
\begin{equation*}
\xi^{\pm}(\lambda,w)=\frac{1}{\pi}\mathrm{Arg}\left(\lim_{\varepsilon\downarrow 0}G(\lambda+i\varepsilon;w)\right)= G^{\pm}(w)\mathbb{H}(\lambda).
\end{equation*}
\end{proof}
\section{Conclusion}\label{conc}
In this work, we have reviewed the proof of a Theorem given by M.G. Kre\u{\i}n on SSFs. The spectral shift function for the Lam\'e equation in the Weierstrass form and the Brioschi-Halphen equation have been calculated. The approach used has been made explicit for researchers to develop a framework for handling other Fuchsian differential equations of this class. This work has demystified the concept and applicability of SSFs to complex differential equations.
\section*{Declarations}
\begin{itemize}
\item\textbf{Compliance with Ethical Standards}: The authors adhered to all ethical standards for the publication of this paper.
\item\textbf{Author's Contribution}: U.S. Idiong reviewed the literature, computed the results and carried out the typesetting of this paper. U.N. Bassey proposed the problem of evaluating SSFs associated with differential operators that define Fuchsian equations.
  \item \textbf{Conflict of Interest}: The authors declare that there are no conflicts of interest in the publication of this paper.
  \item \textbf{Funding}: This research is not funded by any organisation.
 \item \textbf{Data Availability Statements}: Not applicable.
\item\textbf{Ethical Conduct}: This paper is not under consideration for publication in any other journal.
\item\textbf{Acknowledgement}: We acknowledge Fritz Gesztesy, Ram P. Kanwal and other authors whose works have served as a foundation on which these results have been built.
\end{itemize}

\end{document}